\documentclass[a4paper,reqno]{amsart} 

\usepackage{amssymb}
\usepackage[mathcal]{euscript} 
\usepackage[cmtip,all]{xy}
\usepackage{color}
\usepackage{tikz-cd} 
\usepackage{
csquotes}

\DeclareMathOperator{\Hom}{Hom}

\newtheorem{theorem}{Theorem}[section]
\newtheorem{proposition}[theorem]{Proposition}
\newtheorem{lemma}[theorem]{Lemma}
\newtheorem{corollary}[theorem]{Corollary}
\newtheorem{definition}{Definition}

\theoremstyle{remark} 
\newtheorem{remark}[theorem]{Remark}
\numberwithin{equation}{section}

\begin{document}
	
	\title[Indecomposable modules determined by composition factors]{On Artin algebras whose indecomposable modules are determined by composition factors}

	\author[V.~Blasco]{Víctor ~Blasco}
	\thanks{}
	
	\address{Departament of
		Mathematics, Louisiana State University, 70803 Baton Rouge, LA, United States}
	\email{vblasc1@lsu.edu} 	


	
	
	
	\begin{abstract}
	 It was conjectured at the end of the book ``Representation theory of Artin algebras" by M. Auslander, I. Reiten and S. Smal\o{} that an Artin algebra  with the property that its finitely generated indecomposable modules are up to isomorphism completely determined by theirs composition factors is of finite representation type. Examples of rings with this property are the semisimple artinian rings and the rings of the form $\mathbb{Z}_n$. An affirmative answer is obtained for some special cases, namely, the commutative, the hereditary and the radical square zero case. 
	 \end{abstract}
	
	
	\maketitle
	
	\section{Introduction and preliminaries}
	
	Throughout, all the rings will be associative and unital and all the modules, unless otherwise specified, will be right modules. Let $R$ be a ring and $M$ a finite length $R$-module  with composition series	$0\subseteq M_1\subseteq M_2\subseteq...\subseteq M_n=M$. We say that a simple $R$-module $S$ is a composition factor of $M$ if there exists some $j$ s.t. $S\cong M_{j+1}/M_j$, and we denote by $m_{S}(M)$ the multiplicity of $S$ as a composition factor of $M$.  Then we  say that the finite length indecomposable modules over $R$ are completely determined by composition factors if for any two finite length indecomposable modules $M,N$ s.t. $m_{S}(M)=m_{S}(N)$ for all simple $R$-modules $S$ we have that $M\cong N$. By abbreviation we will say that $R$ satisfies the property $\mathfrak{X}$, or just that $R$ satisfies $\mathfrak{X}$,  if the above condition is satisfied.
	\medskip

	Recall that if $R$ is artinian then $R$ has, up to isomorphism, only finitely many simples, say $S_1,...,S_n$, and over artinian rings  the class of the finite length modules coincides with the class of the finitely generated ones. Therefore, given a finitely generated $R$-module $M$, we can associate to it the following $n$-tuple: 
	\begin{equation*}
			\mathfrak{S}(M):=(m_{S_1}(M), m_{S_2}(M),...,m_{S_n}(M))
	\end{equation*}
 In this way we can say,  for any two finitely generated $R$-modules $M,N$, that $m_{S}(M)=m_{S}(N)$ ~ $\forall S$~ simple $R$-module iff $\mathfrak{S}(M)=\mathfrak{S}(N)$. And so $R$ satisfies $\mathfrak{X}$ iff $\mathfrak{S}(M)=\mathfrak{S}(N)$ implies that $M\cong N$ whenever $M$ and $N$ are finitely generated indecomposable $R$-modules. Notice that this is equivalent to saying that two finitely generated indecomposable modules have the same image in the Grothendieck group $K_0(R)$ if and only if they are isomorphic.
	\medskip

	Let $R$ be a commutative artinian ring. A ring $\Lambda$ is called an Artin $R$-algebra if it is a finitely generated algebra over $R$. $\Lambda$ is  said to be of finite representation type if, up to isomorphism, there is only a finite number of pairwise non-isomorphic finitely generated indecomposable modules. It was conjectured in \cite[p. 420]{AR} that, if $\Lambda$ is an Artin $R$-algebra that satisfies $\mathfrak{X}$, then $\Lambda$ must be of finite representation type. In here we will prove three cases of this conjecture: commutative, right hereditary and radical square zero.
	\medskip
	
	We will start with the commutative case in the section 2. First we will see that commutative artinian serial rings, which are precisely the commutative artinian rings of finite representation type, satisfy $\mathfrak{X}$. Then, using results in \cite{AR}  and \cite{F} we will be able to prove the converse. Hence, a commutative artinian ring satisfies $\mathfrak{X}$ if and only if it is of finite representation type.
	\medskip
	
	Section 3 will deal with the property $\mathfrak{X}$ for general abelian categories. We will remind the notion of a full exact embedding of abelian categories and explain how it can help to reduce our problem to ``easier" cases.
	\medskip
	
	The next section will be about algebraic bimodules, which have been extensively studied in \cite{R1}. By reformulating some of the results in \cite{R1} we will show that algebraic bimodules satisfy $\mathfrak{X}$ if and only if they are of finite representation type. This, together will the full exact embeddings mentioned above, will help prove the hereditary case.
	\medskip
	
	Sections 5 and 6 will deal completely with the hereditary case. Firstly, we will reduce the problem to studying certain tensor algebras associated to $K$-species, which in some sense are a generalization of path algebras. Once this is done, the theory of $K$-species of tame representation type developed in \cite{DR1} will help us prove that hereditary Artin algebras of infinite representation type do not satisfy $\mathfrak{X}$. This, together with \cite[Ch. VIII  Corollary 2.4]{AR} will give us a characterization of hereditary Artin algebras of finite representation type, say, a hereditary Artin algebra satisfies $\mathfrak{X}$ if and only if it is of finite representation type. This result seems to be known by some experts on the matter but it cannot be found in the literature available for researchers so its statement and proof is given for completeness.
	\medskip
	
	In the last section, the mentioned above results about hereditary algebras together with the stable equivalence developed in \cite[Ch. X]{AR} will help us show that Artin algebras with radical square zero that satisfy $\mathfrak{X}$ are of finite representation type. We will finish by giving a generic counterexample to the converse of this statement.
	\medskip

	\section{Commutative artinian rings}
	
	We start with the following lemma concerning  commutative artinian uniserial rings:

	\begin{lemma}
		Let $R$ be a commutative artinian uniserial ring. Then $R$ satisfies $\mathfrak{X}$.
	\end{lemma}
	\begin{proof}

		It is well-known that all the indecomposable modules over a commutative artinian uniserial ring $R$ are just all the non-zero cyclic $R$-modules, and so all of them are of the form $R/I$ for some ideal $I$ of $R$.  Therefore, since $R$ has at most one ideal for every length because of uniseriality, we have that $I\neq I'$ implies $l(R/I)\neq l(R/I')$, and so there is at most one indecomposable for every length. Hence, commutative  artinian uniserial rings satisfy $\mathfrak{X}$. 
	\end{proof}
	\medskip
	
	It is easy to check that $\mathfrak{X}$ behaves well under finite products:
	
	\begin{lemma}
		Let $R_1,...,R_n$ be rings. Then $R=\prod\limits_{i=1}^n R_i$ satisfies $\mathfrak{X}$ if and only if each $R_i$ satisfies $\mathfrak{X}$.
	\end{lemma}

	And we get the following:
	
	\begin{theorem}
		Any ring which is a finite product of commutative artinian uniserial rings satisfies $\mathfrak{X}$.
	\end{theorem} 
	\medskip

	Commutative rings which are finite products of commutative artinian uniserial rings are called commutative artinian serial rings and they are characterized in the following way:
	\begin{proposition}
		
		\cite[ Proposition 25.4.6.B]{F} Let $R$ be a 
		commutative artinian ring. Then the following are equivalent:
		\begin{itemize}
			\item[a)] $R$ is serial.
		\item[b)] $R$ is a PIR (i.e. every ideal is principal).
		\item[c)]$R/I$ is self-injective for every proper ideal $I$ of $R$.
		\item[d)] $R$ is of finite representation type.
		\end{itemize}
		
	\end{proposition}
	
	\begin{corollary}
		Let $R$ be  commutative artinian ring of finite representation type. Then $R$ satisfies $\mathfrak{X}$.
	\end{corollary}
	\bigskip
	
	Our aim now is to prove the converse Corollary 2.5. In order to do that we need the following result:
	\begin{theorem} 	\cite[Ch. II Theorem 3.1]{AR} Let $M$ be a finitely generated module over a commutative artinian ring $R$ and call $E:=E_1\oplus ... \oplus E_n$, where $E_i$ is the injective hull of the simple $S_i$ and $S_1,...,S_n$ are all the simple $R$-modules up to isomorphism. Then the multiplicity of every $S_i$ as a composition factor of $M$ is the same than the multiplicity as a composition factor of $\Hom_R(M, E)$.
	\end{theorem}
	
 We can prove now the following:
	\begin{proposition}
		Let $R$ be a commutative artinian ring satisfying $\mathfrak{X}$. Then $R$ is self-injective.
	\end{proposition}
	\begin{proof}
		For, let us recall first that we can write $R=\oplus_{i=1}^n P_i$, where each $P_i$ is a projective indecomposable module. Then, applying the duality $Hom_R(-,E)$ to $R$ we get
		\begin{equation}
			E\cong \Hom_R(R,E)=\Hom_R(\oplus_{i=1}^n P_i,E)\cong \oplus_{i=1}^n Hom_R(P_i,E)
		\end{equation}
		and so each $\Hom_R(P_i,E)$ is indecomposable because $\Hom_R(-,E)$ is a duality, and is injective as it is a direct summand of $E$, which is injective. Moreover, by Theorem 2.6 it has the same composition factors as $P_i$ with the same respective multiplicities. But this just means, by property $\mathfrak{X}$, that $P_i\cong \Hom_R(P_i,E)$, and hence each $P_i$ is injective. We conclude, since a finite direct sum of injective modules is always injective, that $R$ is self-injective.
	\end{proof}
	
	We have the following elementary result, whose corollary follows by Proposition 2.7:
	\begin{proposition}
		Let $R$ be a ring satisfying $\mathfrak{X}$ and let $I$ be any proper ideal  of $R$. Then $R/I$ satisfies $\mathfrak{X}$ too.
	\end{proposition}


	\begin{corollary}
		If $R$ is a commutative artinian  ring satisfying $\mathfrak{X}$ and $I$ is any proper  ideal of $R$, then $R/I$ is self-injective.
	\end{corollary}
	
	
	Thus, by Proposition 2.4 and Corollary 2.9 we get:
	
	\begin{theorem}
		Let $R$ be a commutative artinian ring, then the following are equivalent:\begin{itemize}
			\item[a)] $R$ is serial.
			\item[b)] $R$ is a PIR. 
			\item[c)]$R/I$ is self-injective for every proper ideal $I$ of $R$.
			\item[d)] $R$ is of finite representation type.
			\item[e)] $R$ satisfies $\mathfrak{X}$.
		\end{itemize}
	\end{theorem}
	\bigskip
	\section{Full exact embeddings}
	
	In any abelian category Jordan-H\"older theorem holds for finite length objects. Thus, it makes sense to talk about property $\mathfrak{X}$ for general abelian categories. We say in that case that an abelian category $A$ satisfies $\mathfrak{X}$ on its finite length objects, or just that $A$ satisfies $\mathfrak{X}$.
	
	\begin{definition}
		
		Let $A$, $B$ and $C$ be abelian categories, with $B$ subcategory of $A$, and let $F:C\rightarrow A$ be a functor Then:
		\begin{itemize}
			\item $B$ is a full subcategory of $A$ if for any two objects $X,Y$ in $B$ we have $Hom_B (X,Y)=Hom_A (X,Y)$.
			\item $B$ is an exact subcategory  of $A$ if the inclusion functor from $B$ into $A$ preserves exact sequences (i.e. it is an exact functor). In other words, if $0\rightarrow X\rightarrow Y\rightarrow Z\rightarrow 0$ is a short exact sequence in $B$, then it is also an exact sequence in $A$.
			\item $F$ is called a full exact embedding if it induces an equivalence of (abelian) categories $C\cong F(C)$, and $F(C)$ is a full and exact subcategory of $A$.
		\end{itemize}
	\end{definition}
	\bigskip
	
	Notice that such a functor may not preserve the length of the objects since it may not take simple objects to simple objects. But what is important for us is the following:
	
	\begin{proposition}
		Let $A$ and $B$ be  abelian categories in which every object has finite length. Assume that $B$ does not satisfy  $\mathfrak{X}$  and that there exists a full exact embedding $F:B\rightarrow A$ of abelian categories. Then $A$ does not satisfy $\mathfrak{X}$.
	\end{proposition}
	\begin{proof}

		It is clear that the equivalence of abelian categories  $B\rightarrow F(B)$
		preserves no-$\mathfrak{X}$. It remains to show that indeed $A$ does not satisfy $\mathfrak{X}$. For that, recall that a simple object in $F(B)$ may not be simple in $A$, but at least we can find a composition series in $A$ since every object is a finite length object. This means that, if $M$ and $N$ have the same composition factors with the same multiplicities in $B$, then so does $F(M)$  and $F(N)$ in $A$. Hence, we just need to show that $F(M)$ is an indecomposable object in $A$ if $M$ is indecomposable in $B$. We know that $F(M)$ is indecomposable in $F(B)$ and $F(B)$ is a full and exact subcategory of $A$. Assume that $F(M)=T\oplus H$, with $0\neq T,H$ in $A$, and let $\pi$ be the projection of $F(M)$ onto $H$, viewing it as an endomorphism of $F(M)$. Then $\pi^2=\pi$ is a non-trivial idempotent of End$_A(F(M)$). But $F(B)$ is a full subcategory of $A$ and $F(M)$ is an object in $F(B)$, so End$_B(F(M))=$End$_A(F(M))$. This means that $\pi$ is a  morphism in $F(B)$ and so it is a non-trivial idempotent of End$_B(F(M))$. But this just means that $F(M)$ is not indecomposable in $F(B)$, which is a contradiction.
		\end{proof}
	\begin{corollary}
		Let $R,S$ be artinian rings such that there is a full and exact embedding $F:$ mod-$S\rightarrow$ mod-$R$. If $S$ does not satisfy $\mathfrak{X}$, then $R$ does not satisfy $\mathfrak{X}$.
	\end{corollary}
	
	\begin{remark}
		The main argument used in the above proposition is that any full and exact embedding
		~~~~$F:C\rightarrow D$ between abelian categories where every object is of finite length preserves composition factors in the sense that if $A$ and $B$ are finite length objects in $C$ with the same composition factors and the same multiplicities then $F(A)$ and $F(B)$ are finite length objects in $D$ with the same composition factors and the same multiplicities. This means, since $F$ takes indecomposable objects to indecomposable ones and $C$ and $F(C)$ are equivalent as abelian categories, that $C$ satisfies $\mathfrak{X}$ if and only if $F(C)$ satisfies $\mathfrak{X}$. Hence, in particular if $F$ induces an equivalence of categories then $C$ satisfies $\mathfrak{X}$ if and only if $D$ satisfies $\mathfrak{X}$. An important application of this fact is that $\mathfrak{X}$ is a Morita invariant, i.e.\ a property preserved under equivalences of module categories.
	\end{remark}
	
	\section{Algebraic bimodules}
	Let $F$ and $G$ be two division rings, and let $M$ be a $F$-$G$- bimodule (left module over $F$ and right over $G$) which is finite-dimensional over both division rings. In this case, the upper triangular matrix ring  $R_{M}=\big(\begin{smallmatrix}
		F & M\\
		0 & G
	\end{smallmatrix}\big)$  is a finite-dimensional algebra over some field $K$ if and only if there is a field $K$  such that both $F$ and $G$   
	are finite-dimensional (division)  algebras over $K$ and $K$ operates centrally on $M$. If these conditions are satisfied, $M$ is called an algebraic bimodule. Recall that, in this case,  the dimension of $M$ over $K$ is also finite (see \cite{R1} for details).
	\medskip

	We will assume through this section that $M$ is an algebraic bimodule. It can be shown \cite[ Ch. III Proposition 2.2.]{AR} that in this case the category of finitely generated right modules over $R_{M}$ is equivalent to the following (abelian) category $\mathfrak{L}(M)$:
	\begin{itemize}
		\item The objects are the tuples $(V,W,\phi)$ where $V$ is a f.d. right vector space over $F$, $W$ is a f.d. right vector space over $G$ and $\phi: V\otimes_F M\rightarrow W$ is a $G$-linear map. 
		\item A morphism $\alpha: (V,W,\phi)\rightarrow (V',W',\phi')$ is given by an $F$-linear map $\alpha_F:V\rightarrow V'$ and a $G$-linear map $\alpha_G:W\rightarrow W'$ s.t. $\alpha_G\circ \phi=\phi'\circ (\alpha_F\otimes 1)$.
	\end{itemize}
	\medskip
	
	The tuple $X=(V,W,\phi)$ will be called a representation of $M$ and we will denote by $\dim X$  $=(\dim_F V, \dim W_G)$ its dimension vector.
	\bigskip

	If we denote by $F$ the functor from $\mathfrak{L}(M)$ to mod-$R_{M}$ that induces the mentioned above equivalence, we have the following  (see \cite[Ch. III Proposition 2.3 and Proposition 2.5]{AR}):
	\medskip
	\begin{proposition}
		Let $R_{M}$ and $\mathfrak{L}(M)$ be defined as above. Then:
		
		\begin{itemize}
			
			\item The only simple representations of $R_{M}$ are $(F,0,0)$ and $(0,G,0)$.
			\item $F(X)$ is simple in mod-$R_{M}$ if and only if $X$ is simple in $\mathfrak{L}(M)$.
			\item $0\rightarrow X\rightarrow Y\rightarrow Z\rightarrow 0$ is exact in $\mathfrak{L}(M)$ iff
			
$0\rightarrow F(X)\rightarrow F(Y)\rightarrow F(Z)\rightarrow 0$  is exact in mod-$R_{M}$.
		\end{itemize}
	\end{proposition}
	
	\bigskip
	
	Thus,  since $F$ takes simples to simples and exact sequences to exact sequences, the length of $X$ as an object in $\mathfrak{L}(M)$ is the same as the length of the object $F(X)$ in mod-$R_{M}$; and if $X_1,...,X_n$ are the composition factors of $X$ then $F(X_1),...,F(X_n)$ are the composition factors of $F(X)$. Therefore, we can work with the category  $\mathfrak{L}(M)$ instead of mod-$R_{M}$ and identify modules with representations.
	\bigskip
	
	Associated to an algebraic bimodule $M$ we have the following quadratic form $q(X,Y)$ defined on $\mathbb{Q}^2$ as follows:
	
	\begin{equation}
		q(X,Y)=f_1 X^2+ f_2 Y^2- m XY
	\end{equation}
	where $f_1=\dim_{K} F$, $f_2=\dim_{K} G$ and $m=\dim_{K} M$
	\medskip

	It was shown in \cite{DR1}, \cite{DR2} that $R_M$ is of finite-representation type (i.e. it has only finitely many isoclasses of finitely generated indecomposable modules) iff $q$ is positive definite (i.e. $q(X,Y)>0~ \forall (X,Y)\neq (0,0))$. We can interpret this result just in terms of $a=\dim_{G} M$ and $b=\dim_{F} M$ by defining the (usually non-symmetric) bilinear form  $\tilde{q}$ on $\mathbb{Q}^2$ by
	
	\begin{equation}
		\tilde{q}((X_1,Y_1),(X_2, Y_2))=f_1 X_1 X_2+ f_2 Y_1 Y_2- m X_1 Y_2
	\end{equation}
	whose corresponding quadratic form is just $q$. Then $\tilde{q}$ is given (up to scalar multiplication) by the matrix $\big(\begin{smallmatrix}
		a & -ab\\
		0 & b
	\end{smallmatrix}\big)$ which only depends on $a=\dim_{G} M$ and $b=\dim_{F} M$ (see \cite{R1} for details).
	
	\begin{theorem}\cite[p. 281-282]{R1}
		$R_{M}$ is of finite type iff $ab\leq 3$.
	\end{theorem}
	\medskip
	\begin{remark}
		If $R_M$ is of infinite representation type  (that is, $ab\geq 4$), then for every $x\in \mathbb{Z}_{\geq 0}^2$ s.t. $q(x)\leq 0$ there exists at least one indecomposable representation $X$ of $M$ with $\dim(X)=x$. Furthermore, we can always find  some $x\in \mathbb{Z}_{\geq 0}^2$ of the form $x=(1,s)$  s.t. $q(x)\leq 0$ (see \cite{R1}).
	\end{remark}
	\medskip
	
	\begin{lemma}\cite [Lemma 3.6]{R1}
		In case $K$ is infinite, together with $ab\geq 4$, for each $x=(1,s)$, with $s\in\mathbb{N}$, s.t. $q(x)\leq 0$ we can find an infinite family $\{X_i\}$ of pairwise non-isomorphic indecomposable representations of $M$  with $\dim X_i =x$. If $K$ is finite, we can find at least two of them. 
	\end{lemma}
	
	\begin{remark}
		Let $R$ be an artinian ring satisfying $\mathfrak{X}$. $R$ has, up to isomorphism, only finitely many simple modules  $S_1,...,S_r$, and since it satisfies $\mathfrak{X}$ we have that if $M$ is an indecomposable module of length $m$, then $M$ is completely determined by its composition factors  $n_1*S_1,...,n_r*S_r$, with $m=n_1+...+n_r$. In particular this means that, for a fixed length $m$, the number of indecomposable modules of length $m$ is at most $r!$ times the number of ways of decomposing $m$ as a sum of $r$ non-negative integers, and thus finite. It follows  that, over an artinian ring satisfying $\mathfrak{X}$, the number of indecomposables of any given  length $m$ is always finite. 
	\end{remark}
\medskip
	
	Therefore, from Lemma 4.4 and Remark 4.5 one could guess that algebraic bimodules of infinite representation type do not satisfy $\mathfrak{X}$. That is indeed the case, and for proving it we just need to show that the length, after identifying modules with representations, of two representations with the same dimension vector is the same. In order to do that, let us recall the definitions of subrepresentation and quotient representation:
	
	\begin{definition}
		We say that  $(V_1,W_1,\psi)$ is a subrepresentation of $(V,W,\phi)$ if $V_1$ is a subvector space of $V$, $W_1$ is a subvector space of $W$ and the inclusions maps $\iota: V_1\rightarrow V$, $W_1\rightarrow W$ form a morphism of representations. That is to say, in a sense, that $\psi=\left.\phi\right|_{V_1\otimes M}$ upon identifying image in $V\otimes M$.
	\end{definition}
	
	\begin{definition}
		Let $(V_1,W_1,\psi=\left.\phi\right|_{V_1\otimes M})$ be a subrepresentation of $(V,W,\phi)$. Then the quotient representation $(V/V_1,W/W_1, \delta)$ is given by $\delta: V/V_1\otimes M\rightarrow W/W_1$, with $\delta((a+V_1)\otimes m):=\phi(a\otimes m)+W_1 $.
	\end{definition}
	
	\begin{lemma}

		Let $X=(V,W,\phi)$ be a representation of $M$ with $\dim_{F}V=n$ and $\dim W_{G}=m$. Then in any composition series for $X$ the simple representation $(F,0,0)$ appears $n$ times as a composition factor and the simple representation $(0,G,0)$ appears $m$ times. That is, $l(X)=\dim_{F}V+\dim W_{G}$.
	\end{lemma}
	\begin{proof}
		We know by Proposition 4.1 that the only simple representations of $M$ are $(F,0,0)$ and $(0,G,0)$. If $n=0$ then $X=(0,W,0)\cong (0,G,0)^m$ and the result follows. Otherwise, consider the representation $X_1=(V_1,W,\psi)$ where $V_1$ is a subspace of $V$ with dimension $n-1$ and $\psi=\left.\phi\right|_{V_1\otimes M}$ , where $n=\dim_{F}V$. Then clearly $X_1$ is a subrepresentation of $X$  and we have $X/X_1 \cong (F,0,0)$. If $n-1>0$ we repeat the process with $X_1$ and construct  $X_2$ similarly. We continue this process until reaching $X_n=(0,W,\left.\phi\right|_0)=(0,W,0)$. Then clearly if $m=\dim W_{G}$ then $(0,W,0)\cong (0,G,0)^m$. Thus, we can construct a composition series of the form
		\begin{equation}
			0\subseteq (0,G,0)\subseteq (0,G\oplus G,0)\subseteq ... \subseteq X_n \subseteq ... \subseteq X_1\subseteq X
		\end{equation}
		and so $X$ has length $n+m$, which does not depend on the particular choices of $V$, $W$ and $\phi$. 
	\end{proof}
	\bigskip
	\begin{proposition}
		$R_M$ satisfies $\mathfrak{X}$ iff $ab\leq 3$ iff $R_M$ is of finite representation type.
	\end{proposition}
	\begin{proof}
		Lemma 4.6 tells us in particular that two indecomposable representations have the same dimension vector if and only if they have the same composition factors. This just means by Lemma 4.4 that, if $ab\geq 4$, then we can find at least two non-isomorphic indecomposable representations with the same composition factors; and so $R_M$ does not satisfy $\mathfrak{X}$. On the other hand, it can be shown (see \cite[Theorem 3]{R1}) that for finite type ones (i.e. $ab\leq 3$) we can find at most one indecomposable representation for a given dimension vector. Hence, the result follows.
	\end{proof}

	\section{First reductions on the hereditary case}
	
	\medskip
	It was shown in  \cite[Ch. VIII Corollary 2.4]{AR} that hereditary Artin algebras of finite representation type satisfy $\mathfrak{X}$. Our aim is to prove the converse of this statement.
	
	\begin{definition}
		Let $\Lambda$ be an Artin algebra. Then $\Lambda$ is called basic if we can write $\Lambda=P_1\oplus...\oplus P_n$ with $P_i$'s pairwise non-isomorphic projective indecomposable modules. This is equivalent to saying that $\Lambda/J$ is a product of division rings, where $J$ denotes the Jacobson radical.
	\end{definition}
	\begin{remark}
		
		Suppose that $\Lambda$ is an Artin algebra. We can write $\Lambda=P_{1}^{n_1}\oplus...\oplus P_{m}^{n_m}$ with $n_i >0$ for all $i$ and the $P_i$'s being pairwise non-isomorphic indecomposable projective modules. Let $P=P_1\oplus...\oplus P_m$ and $\Gamma=End_{\Lambda}(P)^{op}$. It can be checked that $\Gamma$ is a basic Artin algebra Morita equivalent to $\Lambda$ (see \cite[p. 35-36]{AR}). Hence, since  $\mathfrak{X}$ is a Morita invariant, we can assume that all the Artin algebras that we are dealing with are basic. 
		
	\end{remark}
	\medskip
	\begin{remark}
		Since $\Lambda$ is an artinian ring, we know that any decomposition of $\Lambda$ into a product of indecomposable rings is finite. So, if it is of infinite representation type then at least one of the indecomposable direct factors must be of infinite representation type, and so
		if we want to prove that $\Lambda$ does not satisfy $\mathfrak{X}$, it is enough to show that such an indecomposable direct factor does not. Thus, we may also assume that all the Artin algebras in the sequel are indecomposable.
	\end{remark}
	
	The following result is the first step towards our aim:

	\begin{theorem}\cite[Section 4 Lemma 1]{DR3}
		Let $R$ be an indecomposable basic semiprimary ring (i.e. its Jacobson radical $J$ is nilpotent and $R/J$ is artinian), and assume there exist primitive orthogonal idempotents $g\neq f$ such that $F=fRf$ and $G=gRg$ are division rings. Then there is a full and exact embedding of the category of finite length modules over the upper triangular matrix ring
		$\big(\begin{smallmatrix}
			F & fJg\\
			0 & G
		\end{smallmatrix}\big)$   into mod-$R$.
	\end{theorem}
	Since every artinian ring is in particular semiprimary this result applies in our context. Assume now that $R$ is furthermore hereditary artinian, and let $e_1,...,e_n$ be a complete set of orthogonal primitive idempotents of $R$. Then we have this classical result:
	\begin{proposition}
		Let $R$ be a hereditary artinian ring and let $e$ be a primitive idempotent of $R$. Then $eRe\cong End_R(Re)$ is a division ring. In particular, if $R$ is also a local ring, then it must be a division ring.
	\end{proposition}
	\medskip
	
	Thus, since division rings are of finite representation type and trivially satisfy $\mathfrak{X}$ we may assume that our hereditary artinian rings are not local. In those cases we can always find distinct primitive orthogonal idempotents $1\neq f,g$. Moreover, from Theorem 5.3, Proposition 5.4 and the fact that the notions of finitely generated and finite length module coincide over artinian rings we get:
	
	\begin{proposition}
		Let $R$ be a basic indecomposable hereditary artinian ring, and let $\{e_1,...,e_n\}$ be a complete set of orthogonal primitive idempotents of $R$. Then $F_i =e_iRe_i$ is a division ring for any $i$, and for any $i\neq j$ we can find a full exact embedding of the category of finitely generated modules over the upper triangular matrix ring
		$\big(\begin{smallmatrix}
			F_i & e_iJe_j\\
			0 & F_j
		\end{smallmatrix}\big)$ 
		into mod-$R$, where $J$ denotes the Jacobson radical of $R$.
	\end{proposition}
	
	\begin{remark}

	In view of the proposition above and Corollary 3.2, if we want to prove that a basic indecomposable hereditary Artin algebra $\Lambda$ of infinite representation type does not satisfy $\mathfrak{X}$ it could be a good strategy to prove that one of the rings of the form 
	$\big(\begin{smallmatrix}
		e_i \Lambda e_i & e_iJe_j\\
		0 & e_j \Lambda e_j
	\end{smallmatrix}\big)$ 
	does not satisfy $\mathfrak{X}$. If such a ring is a finite-dimensional algebra over some field $K$, the results in Section 4 tell us that it does not satisfy $\mathfrak{X}$ precisely whenever $\dim_{e_i \Lambda e_i}e_iJe_j\cdot \dim~ {e_iJe_j}_{ e_j \Lambda e_j}\geq 4$.
\end{remark}
	\medskip
	
	\begin{theorem}\cite[Theorem 4]{RS}
		Let $\Lambda$ be an indecomposable hereditary Artin algebra. Then the center of $\Lambda$ is a field. 
	\end{theorem}
	
	\begin{corollary}
		Let $\Lambda$ be an indecomposable hereditary Artin algebra, $J$ be its Jacobson radical and let $e$ and $f$ be primitive orthogonal idempotents in $\Lambda$. Then $\big(\begin{smallmatrix}
			e\Lambda e & eJf\\
			0 &  f\Lambda f
		\end{smallmatrix}\big)$ is a finite-dimensional algebra over $Z(\Lambda)$.
	\end{corollary}

	\medskip

	\begin{theorem}
		Let $\Lambda$ be a basic indecomposable hereditary Artin algebra, $J$ be its Jacobson radical and let $\{e_1,...,e_n\}$ be a complete set of primitive orthogonal idempotents for $\Lambda$. Denote $F_i=e_i\Lambda e_i$ and ${}_{i} J_{j}=e_iJe_j$ for every $i,j=1,...,n$. If there exist $i\neq j$ s.t. $\dim_{F_i}{}_{i} J_{j}\cdot \dim~{{}_{i} J_{j}}_{F_j}\geq 4$, then $\Lambda$ does not satisfy $\mathfrak{X}$.
	\end{theorem}
	\begin{proof}
		By Corollary 5.8 we know that, for $i\neq j$, every $\big(\begin{smallmatrix}
			F_i& _{i} J_{j}\\
			0 &  F_j
		\end{smallmatrix}\big)$  is a finite-dimensional $Z(\Lambda)$-algebra. This means, by the introductory paragraph to Section 4, that $e_iJe_j$ is an algebraic $F_i,F_j$-bimodule. Also, Proposition 5.5 tells us that we have a full exact embedding of the category of finitely generated modules over the upper triangular matrix ring
		$\big(\begin{smallmatrix}
			F_i & _{i} J_{j}\\
			0 & F_j
		\end{smallmatrix}\big)$ 
		into mod-$\Lambda$.  Furthermore, Proposition 4.7 implies that $\big(\begin{smallmatrix}
			F_i & _{i} J_{j}\\
			0 & F_j
		\end{smallmatrix}\big)$  does not satisfy $\mathfrak{X}$  iff $\dim_{F_i}{}_{i} J_{j}\cdot \dim~{{}_{i} J_{j}}_{F_j}\geq 4$. Finally, from Corollary 3.2 it follows that if $\big(\begin{smallmatrix}
			F_i & _{i} J_{j}\\
			0 & F_j
		\end{smallmatrix}\big)$ does not satisfy $\mathfrak{X}$ then $\Lambda$ does not satisfy $\mathfrak{X}$ neither.
	\end{proof}
	\medskip
	
	We have reduced our question to the study of basic indecomposable hereditary Artin algebras of infinite representation type s.t. $\dim_{F_i}{}_{i} J_{j}\cdot \dim~{{}_{i} J_{j}}_{F_j}\leq 3~~\forall i\neq j$.
	\medskip

	\section{$K-$species and tensor algebras}
	
	Throughout this section $\Lambda$ will denote a hereditary indecomposable basic Artin algebra of infinite representation type s.t. $\dim_{F_i}{}_{i} J_{j}\cdot \dim~{{}_{i} J_{j}}_{F_j}\leq 3~~\forall i\neq j$, where $J$ denotes its Jacobson radical.
	\medskip
	
	We start with the following well-known lemma:
	
	\begin{lemma}
		Let $R$ be a basic hereditary ring, and let $e_1,...,e_n$ be a complete set of primitive orthogonal idempotents of $R$. If for some $i\neq j$ we have $0\neq e_iRe_j$, then $0=e_jRe_i$.
	\end{lemma}
	\begin{proof}
		One can show that, as abelian groups, $e_iRe_j\cong Hom_R(Re_i,Re_j)$. Assume now that $\exists 0\neq f:Re_i\rightarrow Re_j$. Then, since $R$ is hereditary, $Imf$ is projective and so $\tilde{f}:Re_i\rightarrow Imf$ must split. But $Re_i$ is indecomposable because $e_i$ is primitive, and so we must have $Kerf=0$. This just means that $f$ is a monomorphism, and so $l(Re_i)\leq l(Re_j)$. Now, since $R$ is basic, the principal indecomposable modules are pairwise non-isomorphic, that is, $Re_i\cong Re_j\Leftrightarrow i=j$. But, if $0\neq e_jRe_i$ then there must exist a non-zero $g:Re_j\rightarrow Re_i$ which must be mono, and so $l(Re_j)\leq l(Re_i)$. But then, $l(Re_i)= l(Re_j)$ and $f$ is an isomorphism, which is a contradiction. Hence, $0=e_jRe_i$.
	\end{proof}
	\begin{remark}
		We have shown in particular  that $0\neq  Hom_R(Re_i,Re_j)$ implies that any $0\neq f \in Hom_R(Re_i,Re_j)$ is a monomorphism, and so $l(Re_i)<l(Re_j)$ since $Re_i\cong Re_j\Leftrightarrow i=j$ as $R$ is basic. By the same reason, if $l(Re_i)=l(Re_j)$ with $i\neq j$, then $0=Hom_R(Re_i,Re_j)$. Hence, after reordering the principal indecomposable modules $Re_i$'s in such a way that $l(Re_j)< l(Re_i)\Rightarrow j< i$, we can assume that  $e_iRe_j=0$ for every $j>i$ Also,  $J\neq R$ for any unital ring $R$, and so $e_iJe_i\subsetneq e_i R e_i$, but $e_i R e_i$ is a division ring and $e_iJe_i$ is a proper ideal of it, so we must have $0=e_iJe_i$. Hence, we can assume from now on that, if $e_1,...,e_n$ are a complete set of primitive orthogonal idempotents, then $0=e_iJe_j$ for $j\geq i$.
	\end{remark}
	\medskip
	
	Notice that we can write $J=\underset{j<i}{\oplus} e_iJe_j$, and so $J$ has a natural bimodule structure over $F=\underset{i}{\prod} e_i \Lambda e_i=\underset{i}{\prod} F_i$. Now, the condition $\dim_{F_i}e_i J e_j\cdot \dim~{e_i J e_j}_{F_j}\leq 3$ means that either  $e_i J e_j=0$ or $\dim_{F_i}e_i J e_j=1$ or $\dim~{e_i J e_j}_{F_j}=1$. In any case, $e_i J e_j$ must be either $0$ or a simple $F_i,F_j$-bimodule, and so $J$ is a semisimple bimodule over $F$. This means in particular that there exists a bimodule $M$ over $F$ s.t. $J=J^2\oplus M$ as bimodules over $F$.
	\medskip
	
	\bigskip

	We can construct now the tensor ring $T$ associated to $F$ and $J$: 
	
	Write $T=\sum\limits_{n=0}^{\infty} T^{(n)}$, where $T^{(0)}=F$, $T^{(1)}=M$ and, for every $n>1$, consider the $n$-fold tensor product $T^{(n)}= T^{(1)}\otimes_{F} T^{(1)}\otimes_{F}...\otimes_{F} T^{(1)}$~$(n$ times). The multiplication on $T$ is defined through the canonical isomorphism $T^{(i)}\otimes_F T^{(j)}\rightarrow T^{(i+j)}$ and extended distributively. 
	\medskip
	
	
	\begin{proposition}\cite[p. 388-389]{DR1}
		Let $T$ be the tensor ring associated to $F$ and $J$. Then there exists $m\in\mathbb{N}$ s.t. $T^{(m)}=0$. In that case $T$ is a semiprimary ring and $\Lambda\cong T$ as rings.
	\end{proposition}


		\medskip

		Therefore, the class of rings we are studying are tensor rings, which indeed are finite-dimensional algebras over $K=Z(\Lambda)$ since we are assuming that $\Lambda$ is indecomposable and hereditary. The interesting fact is that the representation theory of such a tensor $K$-algebra can be understood in simple terms thanks to the theory of $K$-species that we are going to briefly present now.
		\medskip

		\begin{definition}
			Let $K$ be a field, and let $F_1,...,F_n$ be finite-dimensional division algebras over $K$. Consider now $F_i,F_j$-bimodules ${}_{i} M_{j}$ s.t. $K$ operates centrally on every ${}_{i} M_{j}$ and  every ${}_{i} M_{j}$ is a finite dimensional vector space over $K$. Then $S=(F_i, {}_{i} M_{j})_{1\leq i,j\leq n}$ is called a $K$-species.
		\end{definition}
		\medskip
		
		From $S$ we can derive a quiver $Q(S)$ by writing an arrow $i\rightarrow j$ whenever $0\neq {}_{i} M_{j}$. We call $S$ connected if the underlying graph of $Q(S)$ is connected, and we say that $S$ has no oriented cycle if $Q(S)$ does not have an oriented cycle. The latter condition implies in particular that ${}_{i} M_{i}=0$ for every $i$ and that, if $i\neq j$,  $0\neq {}_{i} M_{j}\Rightarrow 0={}_{j}M_{i}$. 
		\medskip
		\begin{definition}
			A representation $(V_i,\phi_{ij})$ of $S$ is given by a collection of right vector speces over $F_i$ and $F_j$-linear maps $\phi_{ij}: V_{i}\otimes {}_{i} M_{j}\rightarrow V_j$. Such representation is called finite-dimensional provided that all $V_i$ are finite-dimensional vector spaces. A homomorphism between two representations $(V_i,\phi_{ij})$ and $(W_i,\psi_{ij})$ is given by a collection of $F_i$-linear maps $\alpha_i: V_i \rightarrow W_i$ such that $\alpha_j \phi_{ij}=\psi_{ij}(\alpha_{i}\otimes 1)$. 
		\end{definition}
		\medskip
		
		We denote by $\mathfrak{L}(S)$ the category of all finite-dimensional representations of $S$ together with all the homomorphisms between them, which is indeed an abelian category. We say that $S$ is of finite representation type if there are only finitely many isomorphism classes of indecomposable representations in $\mathfrak{L}(S)$. Otherwise, it is said to be of infinite representation type and in that case, as well as for quiver representations, the notions of tame and wild representation type apply by considering a quadratic form defined on $Q(S)$ which coincides with the usual one for quiver representations (see \cite{R1}).
		\medskip
		
		We want to notice that in case all  $F_i=K$ and  all ${}_{i} M_{j}\in \{0,K\}$ a representation of $S$ is just a representation of the quiver $Q(S)$. Also notice that the representations of a  $K$-species of the form $S=(F_1, F_2, {}_{1} M_{2})$ are just the representations of the algebraic bimodule ${}_{1} M_{2}$.
		\medskip
		
		What is interesting for us is that we can associate a tensor $K$-algebra with a connected $K$-species  as follows:
		
		\begin{definition}
			Let $S=(F_i, {}_{i} M_{j})_{1\leq i,j\leq n}$ be a connected $K$-species. Write $F=\Pi_i F_i$ and $M=\underset{i,j}{\oplus} {}_{i} M_{j}$, which is a bimodule over $F$. Then $T(S)=\sum\limits_{n=0}^{\infty} M^{(n)}$,~~~~~~~~~~~~~~~~~~~~~~~~~~~~~~~~~~~~~~~~~~ where $M^{(0)}= F$, $M^{(1)}= M$ and $M^{(n)}=M\otimes_{F}...\otimes_F M$ $(n$ times), is the tensor algebra associated with $S$.
		\end{definition}
		\begin{remark}
			
			By \cite[p. 388]{DR1} $T(S)$ is semiprimary iff $M^{(m)}=0$ for some $m$ iff every sequence $i_1,...,i_n$ with $0\neq {}_{i_k} M_{i_{k+1}}$, for $1\leq k\leq n-1$ is of bounded length. This means in particular that if $T(S)$ is semiprimary then $S$ cannot contain any oriented cycle, since otherwise we can find an infinite path in $S$.
		\end{remark}
		
		The category of finite dimensional representations of a connected $K$-species $S$ without oriented cycles  is equivalent, as an abelian category,  to the category of finitely generated right modules over $T(S)$:
		\begin{proposition}\cite[Proposition 10.1]{DR1} 
			Let $S$ be a connected $K$-species without oriented cycles. Then there is an equivalence of abelian categories between $\mathfrak{L}(S)$ and mod-$T(S)$.
		\end{proposition}
		
		\begin{remark}
			It is not hard to see that if $S$ is connected with no oriented cycles, $T(S)$ satisfies $\mathfrak{X}$ if and only if whenever we have two finite-dimensional indecomposable representations ~  $V=(V_i,\phi_{ij})$ and $W=(W_i,\psi_{ij})$ of $S$ s.t. $\dim~{ V_{i}}_{F_i}=\dim~{ W_{i}}_{F_i}~~\forall i$, then $V\cong W$.
		\end{remark}
		
		\bigskip
		
		We can do the reverse process here, meaning, start with  our specific indecomposable basic hereditary Artin algebra $\Lambda$ of infinite representation type  such that $\dim_{F_i}{}_{i} J_{j}\cdot \dim~{{}_{i} J_{j}}_{F_j}\leq 3~~\forall i\neq j$, where ${}_{i} J_{j}=e_i J e_j$ and $F_i=e_i \Lambda e_i~~\forall i$, which turned out to be a tensor $Z(\Lambda)-$algebra [Proposition 6.4]. We can write $M=\underset{i,j}{\oplus}e_iMe_j=\underset{i,j}{\oplus}{}_{i} M_{j}$, where $M$ was such that $J=M\oplus J^2$ as $F=\underset {i}{\prod} F_i-$modules.  Associate to it now the $K$-species $S(\Lambda)=(F_i, {}_{i} M_{j})_{1\leq i,j\leq n}$. Then $T(S(\Lambda))\cong \Lambda$ by the definition of the tensor algebra associated to $S(\Lambda)$ since we know the tensor algebra structure of $\Lambda$ by Proposition 6.4. Hence, in particular $T(S(\Lambda))$ is artinian and so semiprimary; and so by the remark above $S(\Lambda)$ does not contain any oriented cycle. Also, it is connected since $\Lambda$ was assumed to be indecomposable.
		\medskip
		
		Therefore, by Proposition 6.6 and Remark 3.3 we have:
		\medskip
		
		\begin{theorem}
			There is  an equivalence of abelian categories between $\mathfrak{L}(S(\Lambda))$ and mod-$\Lambda$, and so  $\Lambda$ satisfies $\mathfrak{X}$ if and only if $\mathfrak{L}(S(\Lambda))$ satisfies $\mathfrak{X}$.
		\end{theorem}
		\medskip
		
		Thus, in view of Theorem 5.9 and Theorem 6.7, our discussion has been reduced to study the indecomposable representations over a $K$-species  $S=(F_i, {}_{i} M_{j})_{1\leq i,j\leq n}$ connected and without oriented cycles s.t. $\dim_{F_i}{}_{i} M_{j}\cdot \dim~{{}_{i} M_{j}}_{F_j}\leq 3~~\forall i\neq j$. 
		
		\bigskip
		
		Before continuing we want to recall that if $S$ is a connected $K$-species of wild representation type, then we have a full exact embedding mod-$F\langle X,Y\rangle\rightarrow \mathfrak{L}(S)$ for some field $F$ \cite[Theorem 2]{R1}. In that case, $F\langle X,Y\rangle$ never satisfies $\mathfrak{X}$. One way of seeing this is by considering the quotient ring $A=F\langle X,Y\rangle/(X^2,Y^2,XY-YX)$. $A$ is a commutative artinian ring but it is not PIR since $0\neq I=(X,Y)$ is a proper ideal of $A$. Then, by Theorem 2.10 it follows that $A$ does not satisfy $\mathfrak{X}$, and hence $F\langle X,Y\rangle$ does not satisfy $\mathfrak{X}$ neither. Thus, wild $K$-species do not satisfy $\mathfrak{X}$, and so we can assume that our $K$-species of infinite representation type have all tame type.
		\medskip
		
		A complete characterization of the connected $K$-species with  no oriented cycles which are of tame type has been done in \cite{DR2}, and a complete list of all the possible cases appears in there. So an inspection case by case of the possible $K$-species that can satisfy our condition could be done in order to see that they cannot satisfy $\mathfrak{X}$. However, a more general theorem applies to all connected $K$-species with no oriented cycles  of tame representation type.
		\medskip
		
		Before stating the theorem we need the following definitions:
		
		\begin{definition}
			Let $M$ be an algebraic bimodule over $F$ and $G$, where $F$ and $G$ are division rings. Then $M$ is said to have type $\tilde{A}_{ab}$ if $\dim_{F} M=a$ and $\dim~ M_G=b$.
		\end{definition}
		\begin{definition}\cite[7.1]{R1} 
			Let $M$ be an algebraic bimodule s.t. $ab=4$, $a\leq b$.  Given an indecomposable representation $V=(V_1,V_2,\phi)$ its defect $\sigma V$ is defined as $\sigma V=\dim_F V_1-\dim {V_{2}}_{G}$ in case $2=a=b$ and by  $\sigma V=2\cdot \dim_F V_1-\dim {V_2}_G$ in case $a=1$ and $b=4$. A representation  is called regular if it is the direct sum of indecomposable representations with zero defect. We denote by $\mathfrak{r}(M)$ the full subcategory of all regular (finite-dimensional) representations of $M$, which is indeed an abelian exact subcategory of $\mathfrak{L}(M)$ closed under extensions.
		\end{definition} 
		
		\begin{remark}
			The definition in \cite[7.1]{R1} is stated for what we call affine bimodules with constant dimension, but an algebraic bimodule $M$ s.t. $ab=4$ with $a\leq b$ is such (see \cite[6.1, 6.4]{R1}).
		\end{remark}
		
		We are ready now to state the theorem:
		\begin{theorem}
			\cite[Theorem 5.1]{DR2} Let $S$ be a connected $K$-species with no oriented cycles and of tame representation type. Then there exists an algebraic bimodule $M$ of type $\tilde{A}_{14}$ or $\tilde{A}_{22}$ and a full exact embedding $\mathfrak{F}: \mathfrak{r}(M)\rightarrow \mathfrak{L}(S)$.
		\end{theorem}
		\begin{corollary}
			If $S$ be a connected $K$-species with no oriented cycles and of tame representation type. Then $\mathfrak{L}(S)$ does not satisfy $\mathfrak{X}$.
		\end{corollary}
		\begin{proof}
			For the case $2=a=b$, the billinear form $\tilde{q}$ associated to $M$ defined in  Section 4 has the form (up to scalar multiplication) 
			$\big(\begin{smallmatrix}
				2 & -4\\
				0 & 2
			\end{smallmatrix}\big)$. In that case for $x=(1,1)$ we have $\tilde{q}((1,1))=0$, and so there is an indecomposable representation $V$ with dimension vector $(1,1)$ which  satisfies $\sigma V=0$. Thus, it is a regular representation and we know by Lemma 4.4 that there are at least two non-isomorphic indecomposable (regular) representations with dimension vector $(1,1)$.
			
			For the case of $a=1,b=4$ we have $\tilde{q}((1,2))=0$, so the same argument holds since every indecomposable representation with dimension vector $(1,2)$ is regular in that case.
			\medskip
			
			In any case, this just means that $\mathfrak{r}(M)$ does not satisfy $\mathfrak{X}$. Therefore, by Theorem 6.9 and Proposition 3.1  we conclude that $\mathfrak{L}(S)$ cannot satisfy $\mathfrak{X}$.
		\end{proof}
		\medskip
		
		Thus, by Theorem 6.7 we get the following:
		
		\begin{proposition}
			Let $\Lambda$ be an indecomposable basic hereditary Artin algebra of infinite representation type s.t. $\dim_{F_i}{}_{i} J_{j}\cdot \dim~{{}_{i} J_{j}}_{F_j}\leq 3~~\forall i\neq j$, where $F_i=e_i \Lambda e_i$, ${}_{i} J_{j}=e_i J e_j$ and $J$ is the Jacobson radical of $\Lambda$. Then $\Lambda$ does not satisfy $\mathfrak{X}$.
		\end{proposition}
		\medskip
		
		Proposition 6.11 and Theorem 5.9 imply:
		\begin{theorem}
			Let $\Lambda$ be a basic indecomposable hereditary Artin algebra of infinite representation type. Then $\Lambda$ does not satisfy $\mathfrak{X}$.
		\end{theorem}
		Finally, \cite[Ch. VIII Corollary 2.4]{AR}, together with the Remarks 5.1 and 5.2 imply:
		
		\begin{theorem}
			The following are equivalent for a hereditary Artin algebra $\Lambda$:
			\begin{itemize}
				\item[a)] $\Lambda$ is of finite representation type.
				\item[b)]$\Lambda$ satisfies $\mathfrak{X}$.
			\end{itemize}
		\end{theorem}
		
		\bigskip
		\section{Artin algebras with radical square zero}
		
		Our aim in this section is to prove that Artin algebras with radical square zero that satisfy $\mathfrak{X}$ are of finite representation type. We will also show that the converse of that statement is not correct.
		
		\medskip

		Let $\Lambda$ be an Artin algebra. Recall the category $\underline{mod}~\Lambda$: 
		\begin{itemize}
			\item Its objects are the objects of mod-$\Lambda$.
			\item $\underline{Hom}~(A,B):=\frac{Hom(A,B)}{P(A,B)}$ for any $A,B$ in mod-$\Lambda$, where $P(A,B)$ denotes the homomorphisms $f:A\rightarrow B$ that factor through a projective module, i.e. $\exists Q$ projective module and $g:A\rightarrow Q$, $h:Q\rightarrow B$ s.t. $f=h\circ g$.
		\end{itemize}
		\begin{definition}
			We say that two Artin algebras $\Lambda$ and $\Gamma$ are stably equivalent if there exists an equivalence of categories $F:\underline{mod}~\Lambda\rightarrow \underline{mod}~\Gamma$.
		\end{definition}
		
		\begin{remark}
			If two Artin algebras are Morita equivalent then the same functor that induces the equivalence induces a stable equivalence between them since projective modules and maps that factor though projectives are preserved under such equivalences. But it is not true that two stably equivalent algebras are necessarily Morita equivalent. 
		\end{remark}
	
		\begin{proposition}\cite[Ch. X Proposition 1.1]{AR} Let $\Lambda$ and $\Gamma$ be stably equivalent Artin algebras. Then $\Lambda$ is of finite representation type if and only if $\Gamma$ is of finite representation type.
		\end{proposition}
		\medskip

		The following result follows as a corollary of \cite[Ch. III Proposition 2.7]{AR}:
		\begin{proposition}
			Suppose $T$ and $U$ are semisimple Artin $R$-algebras and let $M$ be a non-zero $T,U$-bimodule where $R$ acts centrally and $M$ is finitely generated over $R$. Then $\Delta=\big(\begin{smallmatrix}
				T & M\\
				0 & U
			\end{smallmatrix}\big)$ is an hereditary Artin $R$-algebra.
		\end{proposition}
		
		\begin{corollary}
			Let $\Lambda$ be any Artin $R$-algebra such that $J^2=0$, where $J$ denotes its Jacobson radical, and let $\Gamma$ be the upper triangular matrix algebra $\big(\begin{smallmatrix}
				\Lambda/J & J\\
				0 & \Lambda/J
			\end{smallmatrix}\big)$. Then $\Gamma$ is hereditary.
		\end{corollary}
		\medskip
		
		By \cite[Ch. III Proposition 2.2]{AR} it follows that we can see the modules over $\Gamma=\big(\begin{smallmatrix}
			\Lambda/J & J\\
			0 & \Lambda/J
		\end{smallmatrix}\big)$ as tuples $(A,B,\phi)$ where $A$ is a module over $\Lambda/J$, $B$ is a module over  $\Lambda/J$ and $\phi: A\otimes J\rightarrow B$ is a $\Lambda/J$-linear map. Then the same results that we stated for example in Prop. 4.1 apply in our context since both are direct consequences of the results in \cite[Ch. III]{AR}. This tells us in particular that any simple module over $\Gamma$ is either of the form $(S,0,0)$ or $(0,S,0)$, where $S$ is some simple module over $\Lambda/J$ (and therefore over $\Lambda$, since $J$ annihilates simples); and also that if $A=T_1\oplus...\oplus T_n$, and $B=H_1\oplus...\oplus H_m$ with $H_i, T_j$ simple $\Lambda/J$ modules then
		\begin{equation}
			0\subseteq (0,H_1,0)\subseteq (0,H_1\oplus H_2,0)\subseteq...\subseteq (0,B,0)\subseteq
		\end{equation}
		\begin{equation}
			\subseteq (T_1,B,\phi|_{T_1})\subseteq (T_1\oplus T_2,B,\phi|_{T_1\oplus T_2})\subseteq ...\subseteq (A,B,\phi)
		\end{equation}
		is a composition series for $(A,B,\phi)$, and  $\{(T_j,0,0)\}\cup \{(0,H_i,0)\}$ are its composition factors. Define now the functor $F: mod$-$\Lambda\rightarrow mod$-$\Gamma$:
		\begin{itemize}
			\item $F(M)=(M/JM, JM, f)$ where $f: J\otimes_{\Lambda/J} M/JM\rightarrow JM$ is  given by  $f(x\otimes \tilde{m})=xm$.
			
			\item For morphisms, recall that if $g:M\rightarrow N$ is a morphism in $mod$-$\Lambda$ then $g(JM)=Jg(M)\subseteq JN$ and so the restriction of $g$ to $JM$ has image in $JN$. Call the restriction $\tilde{g}$. Then $g$ induces a morphism $g':M/JM\rightarrow N/JN$. Define $F(g)=(g',\tilde{g})$.
			
		\end{itemize}
		\bigskip
		
		\begin{remark}
			It follows from the comments above that $l(M)=l(F(M))$ and that, if $T_1,...,T_n$ are the composition factors of $JM$ and $H_1,...,H_m$ are the composition factors of $M/JM$, then $(0,T_1,0),...,(0,T_n,0),(H_1,0,0),...,(H_m,0,0)$ are the composition factors of $F(M)$. 
		\end{remark}
		\medskip
		
		\begin{theorem}\cite[Ch. X Lemma 2.1, Theorem 2.4]{AR} Let $F$ be defined as above. Then $F$ induces a stable equivalence $\underline{mod}$-$\Lambda\rightarrow \underline{mod}$-$\Gamma$, which we denote also by $F$, and it has the following properties:
			\begin{itemize}
				\item[a)] $M$ in mod-$\Lambda$ is indecomposable if and only if $F(M)$ is indecomposable.
				\item[b)] $M$ and $M'$ are isomorphic if and only if $F(M)$ and $F(M')$ are isomorphic.
				\item[c)] Let $X=(A,B,\phi)$ in mod-$\Gamma$. Then $X\cong F(M)$ for some $M$ in mod-$\Lambda$ if and only if $\phi$ is an epimorphism.
				\item[d)] Let $X=(A,B,\phi)$ indecomposable in mod-$\Gamma$ . Then $X\ncong F(M)$ for any $M$ in mod-$\Lambda$ if and only if $X\cong (0,S,0)$, with $S$ some simple $\Lambda$-module.
			\end{itemize}
			
		\end{theorem}
		\bigskip
		
		We are in a good position now to establish our main theorem of this section:
		\begin{theorem}
			Let $\Lambda$ be an Artin algebra with $J^2=0$ that satisfies $\mathfrak{X}$, ~~and let $\Gamma=\big(\begin{smallmatrix}
				\Lambda/J & J\\
				0 & \Lambda/J
			\end{smallmatrix}\big)$. Then $\Gamma$ satifies $\mathfrak{X}$.
		\end{theorem}
		\begin{proof}
			We need to show that if $X=(A,B,\phi)$ and $Y=(A',B',\psi)$ are two indecomposable modules of $\Gamma$ with the same composition factors, then they are isomorphic. Since $\Lambda/J$ is semisimple, then both $A$ and $B$ are semisimple modules and so we can write $A\cong T_1\oplus...\oplus T_n$, and $B\cong H_1\oplus...\oplus H_m$ with $H_i, T_j$ simple $\Lambda/J$-modules~ (and so simple $\Lambda$-modules). We know that in that case the composition factors of $X$ are $\{(T_j,0,0)\}\cup \{(0,H_i,0)\}$. Hence, since $X$ and $Y$ have the same composition factors and $\Lambda/J$ is semisimple  we must also have  $A'\cong T_1\oplus...\oplus T_n$, and $B'\cong H_1\oplus...\oplus H_m$. Thus, $A\cong A'$ and $B\cong B'$ as $\Lambda/J$-modules (and so as $\Lambda$-modules).
			\medskip
			
			Now, by Theorem 7.6 we know that the only indecomposable modules over $\Gamma$ which are not of the form $F(M)=(M/JM,JM,f)$ with $M$ indecomposable over $\Lambda$, are the ones of the form $(0,S,0)$ with $S$ simple module over $\Lambda$. Therefore, since those modules are simple they are trivially determined by their composition factors, we just need to check that the ones of the form $F(M)=(M/JM,JM,f)$ with $M$ indecomposable over $\Lambda$ are determined by their composition factors.
			\medskip
			
			For, let $M$ and $N$ indecomposable modules over $\Lambda$  s.t. $F(M)=(M/JM,JM,f)$ and $F(N)=(N/JN,JN,f)$ have the same composition factors. This means, by the first paragraph above, that $M/JM\cong N/JN$ and $JM\cong JN$ as $\Lambda$-modules. Hence, since the composition factors of $M$ are the union of the composition factors of $JM$ and $M/JM$, it follows that $M$ and $N$ have the same composition factors, and so they must be isomorphic since they are indecomposable and  $\Lambda$ satisfies $\mathfrak{X}$. But then $F(N)\cong F(M)$ and we conclude that $\Gamma$ satisfies $\mathfrak{X}$.
			
		\end{proof}
		
		\begin{corollary}
			Any Artin algebra $\Lambda$ with $J^2=0$ that satisfies $\mathfrak{X}$ must be of finite representation type.
		\end{corollary}
		\begin{proof}
			By Theorem 7.7 we know that $\Gamma=\big(\begin{smallmatrix}
				\Lambda/J & J\\
				0 & \Lambda/J
			\end{smallmatrix}\big)$ also satisfies $\mathfrak{X}$, and  $\Gamma$ is hereditary by Corollary 7.4. Therefore, by Theorem 6.14, $\Gamma$ must be of finite representation type. But then, by Proposition 7.2,  $\Lambda$ is of finite representation type.
		\end{proof}
		\bigskip
		
		We would like to finish this section by showing a counterexample for the converse of Corollary 7.8.	Let $Q$ be the following quiver
	
	
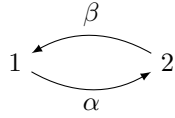
\begin{center}
		\begin{tikzpicture}
			\node (a) at (5,0) {1};
			\node (b) at (7,0) {2};
			\draw[-latex,bend right]  (a) edge node [below] {$\alpha$} (b) ;
			\draw[-latex,bend right]  (b) edge node [above] {$\beta$} (a) ;
		\end{tikzpicture}
	\end{center} and let $K$ be any field. Consider the  path algebra $KQ$, and let $I$ be the ideal generated by $\alpha\cdot\beta$ and $\beta\cdot\alpha$. Let $A$ be the algebra $KQ/I$. Recall that there is an equivalence of categories between mod-$A$ and $\mathfrak{rep}_{K}Q$, where $\mathfrak{rep}_{K}Q/I$  is  the  full  subcategory of representations of $Q$ whose objects are representations of $Q$ 
		
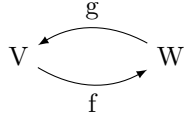
		\begin{center}

		\begin{tikzpicture}
			\node (a) at (5,0) {V};
			\node (b) at (7,0) {W};
			\draw[-latex,bend right]  (a) edge node [below] {f} (b) ;
			\draw[-latex,bend right]  (b) edge node [above] {g} (a) ;
		\end{tikzpicture}
	\end{center}s.t. $0=fg=gf$. We can therefore identify modules with such representations. Consider now the indecomposable representations $P_1$

		
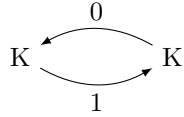
\begin{center}

		\begin{tikzpicture}
			\node (a) at (5,0) {K};
			\node (b) at (7,0) {K};
			\draw[-latex,bend right]  (a) edge node [below] {1} (b) ;
			\draw[-latex,bend right]  (b) edge node [above] {0} (a) ;
		\end{tikzpicture}
	\end{center}and $P_2$
		
		
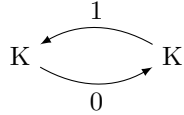
\begin{center}

		\begin{tikzpicture}
			\node (a) at (5,0) {K};
			\node (b) at (7,0) {K};
			\draw[-latex,bend right]  (a) edge node [below] {0} (b) ;
			\draw[-latex,bend right]  (b) edge node [above] {1} (a) ;
		\end{tikzpicture}
	\end{center}These are representations of $Q/I$ since $0=1\cdot 0=0\cdot 1$. They are non-isomorphic because for being isomorphic we should have, after identifying invertible linear maps from $K$ to $K$ with non-zero scalars in $K$,  non-zero elements $a,b\in K$ s.t. $0\cdot a=b\cdot 1$, which is impossible. Consider now the simple subrepresentation $S_2$ of $P_1$
		
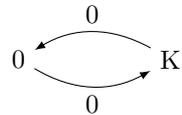
\begin{center}

		\begin{tikzpicture}
			\node (a) at (5,0) {0};
			\node (b) at (7,0) {K};
			\draw[-latex,bend right]  (a) edge node [below] {0} (b) ;
			\draw[-latex,bend right]  (b) edge node [above] {0} (a) ;
		\end{tikzpicture}
	\end{center}
		and the simple subrepresentation  $S_1$ of $P_2$
		\begin{center}

		\begin{tikzpicture}
			\node (a) at (5,0) {K};
			\node (b) at (7,0) {0};
			\draw[-latex,bend right]  (a) edge node [below] {0} (b) ;
			\draw[-latex,bend right]  (b) edge node [above] {0} (a) ;
		\end{tikzpicture}
	\end{center}
		\medskip
		
		One can check that $P_1/S_2\cong S_1$ and $P_2/S_1\cong S_2$, and so both $P_1$ and $P_2$ have the same composition factors. But this just means that $A$ does not satisfy $\mathfrak{X}$.
		\medskip
		
		It is not hard to check that the only indecomposable representations of $Q/I$ are $S_1,S_2,P_1,P_2$ and so $A$ is of finite representation type.  Therefore, the length of any indecomposable module over $A$ is at most two. Furthermore, we can write $A=A_1\oplus A_2$  where $A_i$ is the projective indecomposable module corresponding to $P_i$ under the equivalence of categories. We necessarily have that $0\subseteq JA_i\subseteq  A_i$ is a composition series for $A_i$ and so in particular $J^2A_i=0$ since $J^2A_i$ is properly contained in $JA_i$. Hence, $J^2=J^2A=J^2 A_1\oplus J^2A_2=0$. Therefore, we get:
		\begin{proposition}
			Let $K$ be any field. Then there exists a finite-dimensional $K$-algebra $A$ with radical square zero and  of finite representation type which does not satisfy $\mathfrak{X}$.
		\end{proposition}

		\begin{remark}
			The example we have seen can be generalized to an $n$-oriented cycle  $Q_n$ by letting $I_n$ be the ideal generated by all the paths of length $n$ and considering the algebra $Q_n/I_n$.
			
		\end{remark}
		
		\section*{Acknowledgements}
	The results and ideas developed in this article have been deeply influenced by the work of M. Auslander, C.M. Ringel and their respective collaborators. The author feels in debt with them.

\end{document}